\documentclass[12pt]{amsart}
\usepackage{amsmath,amssymb,latexsym}
\usepackage{fullpage}
\usepackage{enumerate}

\theoremstyle{plain}
\newtheorem{theorem}{Theorem}
\newtheorem{proposition}{Proposition}
\newtheorem{lemma}{Lemma}

\newtheorem{conjecture}{Conjecture}
\newtheorem{assumption}{Assumption}

\def\A{\mathbb{A}}
\def\C{\mathbb{C}}
\def\Ind{\text{Ind}}

\begin{document}

\title{Theta Functions on Covers of Symplectic Groups}
\dedicatory{To Freydoon Shahidi on his 70th birthday}
\author{Solomon Friedberg}
\author{David Ginzburg}
\email{ (corresponding author) solfriedberg@gmail.com}
\thanks{This work was supported by the BSF,
grant number 2012019, and by the NSF, grant number 1500977 (Friedberg).}
\begin{abstract} We study the automorphic theta representation $\Theta_{2n}^{(r)}$ on the $r$-fold cover of the symplectic group $Sp_{2n}$.  This representation is obtained from the residues of Eisenstein series on this group.
If $r$ is odd, 
$n\le r <2n$, then under a natural hypothesis on the theta representations, we show that 
 $\Theta_{2n}^{(r)}$ may be used to construct a generic representation
 $\sigma_{2n-r+1}^{(2r)}$ on the $2r$-fold cover of $Sp_{2n-r+1}$.  Moreover, when $r=n$ the
Whittaker functions of this representation attached to factorizable data 
 are factorizable, and the unramified local factors may be computed in terms of $n$-th order Gauss sums. 
If $n=3$ we prove these results, which in that case pertain to the six-fold cover of $Sp_4$,  unconditionally.
We expect that in fact the representation constructed here, $\sigma_{2n-r+1}^{(2r)}$, is  precisely $\Theta_{2n-r+1}^{(2r)}$; that
is, we conjecture relations between theta representations on different covering groups.
\end{abstract}

\maketitle

\section{Introduction}

The classical metaplectic group is a double cover of a symplectic group, and may be defined over a local field or the ring of adeles of a number field.  This group arises in the study of theta functions, 
which may be constructed directly as sums over the global rational points of an isotropic subspace.   This construction 
of theta functions appears
to be special to the double cover case.  

To generalize, let $r>1$ be any integer, and let
$F$ be a number field with a full set of $r$-th roots of unity.  Then there is an $r$-fold cover of the symplectic
group  $Sp_{2n}(\A)$, where $\A$ is the adeles of $F$, 
denoted $Sp_{2n}^{(r)}(\A)$.  When $r>2$ one may construct theta functions for such a group, but only indirectly -- theta
functions are obtained 
as the multi-residues of the minimal parabolic (Borel) Eisenstein series on $Sp_{2n}^{(r)}(\A)$ at 
the ``right-most" (suitably normalized) pole.  (This construction obtains when $r=2$ as well.)
 As such, these theta functions are automorphic forms and give an
automorphic representation $\Theta_{2n}^{(r)}$, called the theta representation.

Since for a general cover the theta functions may only be defined via residues, it is considerably more challenging
to determine basic information such as their Whittaker or other Fourier
coefficients in this situation.  Indeed, the Fourier coefficients of the theta function
have not been determined for covers of $SL_2$ of any degree higher than $4$.  And when $n>1$ very little is known.  
Moreover, except in the double cover case there are few examples of any automorphic functions on such a group, there 
 is no information about when their Whittaker coefficients might be factorizable (one expects rarely),  and it is not clear
 to what extent one might expect relations between theta functions or other automorphic
forms on different covers of different groups. 

 In a first example of such information, the authors \cite{F-G} have established
relations between
the Whittaker coefficients of certain automorphic functions on different covering groups in two specific cases, related to the conjectures 
of  Patterson and Chinta-Friedberg-Hoffstein concerning the Fourier coefficients of theta functions on 
covers of $GL_2$ .  The key to doing so was to adapt
descent methods, originally used by Ginzburg-Rallis-Soudry \cite{G-R-S1} in the context of 
algebraic groups or double covers, to higher degree covering groups.  In particular, the treatment of
the adelic version of the conjecture
of Chinta-Friedberg-Hoffstein  relied on the study of the theta function on the three-fold cover of $Sp_4$.

 In this work we investigate theta functions and descent integrals on covers of symplectic groups of arbitrary rank.  
 As we shall explain, for $r$ odd a descent construction allows one to pass from
the theta representation $\Theta_{2n}^{(r)}$ to a representation $\sigma_{2n-r+1}^{(2r)}$ realized
inside the space of automorphic square-integrable functions on 
$Sp^{(2r)}_{2n-r+1}(F)\backslash Sp^{(2r)}_{2n-r+1}(\A)$.  We expect that this representation is in fact the
theta representation $\Theta_{2n-r+1}^{(2r)}$ (Conjecture~\ref{conjecture-new-strong}), and we show that $\sigma_{2n-r+1}^{(2r)}$ has non-zero
projection to this representation.
In Theorem~\ref{th1} below we establish that if $r$ is odd, 
$n\le r <2n$, then the  representation $\sigma_{2n-r+1}^{(2r)}$ is generic. Moreover, if $r=n$, 
we show that the Whittaker coefficients of 
$\sigma_{n+1}^{(2n)}$ arising from factorizable inputs are factorizable.
As a first new case, this is true of the Whittaker coefficients
of the descent representation on the six-fold cover of $Sp_4$.
 Then in Theorem~\ref{th2} we show that the unramified local 
contributions to the Whittaker coefficients in the factorizable
case may be expressed as sums of Whittaker coefficients of the theta representation on the $n$-fold cover of $GL_n$.  
These coefficients are $n$-th order Gauss sums by the work of Kazhdan and Patterson \cite{K-P}.   Thus we exhibit a new
class of automorphic forms on covering groups whose Whittaker coefficients are factorizable with algebraic factors at good places.

These results are proved on
two conditions.  One concerns the specific unipotent orbit attached to $\Theta_{2n}^{(r)}$ (Conjecture~\ref{conj1}).  
This is an object of on-going study \cite{S1}; in the last Section we give a sketch of the proof of this Conjecture in the
first new case, $n=r=3$.  A
second concerns the precise characters within a given unipotent orbit which support a nonzero Fourier coefficient
(Assumption~\ref{assume1}).  This is not needed in all cases, and in particular is not needed in the factorizable case
$\sigma_{n+1}^{(2n)}$.   
Our results concerning Conjecture 1 are thus sufficient to establish unconditionally that
the Whittaker function of the descent to the six-fold cover of $Sp_4$ is Eulerian for factorizable data, and to compute its unramified 
local factor in terms of cubic Gauss sums.  We expect that this descent is in fact the theta function on the six-fold cover of $Sp_4$.

\section{Definition and Properties  of the Theta Representation}\label{def}
Let $n$ and $r$ be two natural numbers.  Let $F$ be a number field containing a full set $\mu_r$ of $r$-th roots of unity,
and let $\A$ be the adeles of $F$.  Let $Sp_{2n}$ denote the symplectic group consisting of the $2n\times 2n$ matrices leaving
invariant the standard symplectic form $<x,y>=\sum_{i=1}^{n} (x_iy_{2n-i+1}-x_{n+i}y_{n-i+1})$.
Denote by $Sp_{2n}^{(r)}(\A)$ the metaplectic $r$-fold cover of the symplectic group $Sp_{2n}(\A)$.  This group 
consists of pairs $(g,\zeta)$ with $g\in Sp_{2n}(\A)$ and $\zeta\in\mu_r$; it
may be
obtained in the standard way from covers of the local groups $Sp_{2n}(F_\nu)$ as $\nu$ ranges over the places of $F$,
identifying the copies of $\mu_r$.  The multiplication in the local group is determined by a cocyle $\sigma$. 
If $\sigma_{BLS}$ is the 2-cocycle of
Banks-Levy-Sepanski \cite{B-L-S} for $SL_{2n}(F_\nu)$ and $w$ is the permutation such that the conjugate of $Sp_{2n}$
by $w$ preserves the symplectic form $<x,y>'=\sum_{i=1}^n x_i y_{n+i}-y_ix_{n+i}$, then we take
$\sigma(g_1,g_2)=\sigma_{BLS}(wg_1w^{-1}, wg_2w^{-1})$.

We shall be concerned with the theta representation defined on the group $Sp_{2n}^{(r)}(\A)$, denoted $\Theta_{2n}^{(r)}$. This representation is defined as 
the space spanned by the residues of the Borel Eisenstein series on $Sp_{2n}^{(r)}({\A})$, similarly to the definition for the general linear
group in \cite{K-P}, p.\ 118.  Our first task is to give a brief account of this construction.

Let $B_{2n}$ denote the standard Borel subgroup of $Sp_{2n}$, and let $T_{2n}$ denote the maximal torus of $Sp_{2n}$ which is a subgroup of $B_{2n}$.  Working either locally at a completion of $F$ or globally over $\A$, if $H$ is any algebraic subgroup of $Sp_{2n}$, we let $H^{(r)}$ denote its inverse image in $Sp_{2n}^{(r)}$.  We will call this inverse image parabolic or unipotent if $H$ has this property.
Let  $Z(T_{2n}^{(r)})$ denote the center of $T_{2n}^{(r)}$. 
Suppose first that $r$ is odd.
Given a character $\mu$ of $T_{2n}$, we may use it to define a genuine character,
again denoted $\mu$, of $Z(T_{2n}^{(r)})$.  
(The notion of genuine depends on an embedding of the group of $r$-th roots of unity $\mu_r$ into $\C^\times$.  We will fix this and omit it from the notation.)
Extending it trivially to any maximal abelian subgroup of $T_{2n}^{(r)}$ and then inducing up, we obtain a representation of $Sp_{2n}^{(r)}$ which we denote by $\Ind_{B_{2n}^{(r)}}^{Sp_{2n}^{(r)}}\mu$. This representation is determined uniquely by the choice of $\mu$ and this procedure may be carried out both locally and globally.   Here we consider
unnormalized induction, and shall include the modular character $\delta_{B_{2n}}^{1/2}$ when we require normalized induction.

Let $s_i$ be complex variables, and let $\mu$ be the character of $T_{2n}$ given by
\begin{equation*}\label{mu1}
\mu(\text{diag}(a_1,\ldots,a_n,a_n^{-1},\ldots, a_1^{-1}))=|a_1|^{s_1}\cdots |a_n|^{s_n}.
\end{equation*}
If this construction is carried out over a local field, then $|a|$ denotes the normalized local absolute value, while if it is
carried out over $\A$, then $|a|$ denotes the product of these over all places of $F$.
One may form the induced representation $\Ind_{B_{2n}^{(r)}({\A})}^{Sp_{2n}^{(r)}({\A})}\mu\delta_{B_{2n}}^{1/2}$, and
for each vector in this space, one may construct the multi-variable Eisenstein series $E(h,s_1,\ldots,s_n)$ defined on the group $Sp_{2n}^{(r)}({\A})$. Computing the constant term as in \cite{K-P}, Prop.\ II.1.2, we deduce that the poles of the partial intertwining operator associated with the long Weyl element of $Sp_{2n}$ are determined by
\begin{equation}\label{mu2}
\frac{\prod_{i<j} [ \zeta_S(r(s_i-s_j))\zeta_S(r(s_i+s_j))] \prod_i\zeta_S(rs_i)}{\prod_{i<j} [ \zeta_S(r(s_i-s_j)+1)\zeta_S(r(s_i+s_j)+1)] \prod_i\zeta_S(rs_i+1)}.
\end{equation}
Here $S$ is a finite set of all places including the archimedean places and all  finite places $\mu$ such that $|r|_\nu\neq1$
and $\zeta_S(s)$ is the partial global zeta function.

The expression \eqref{mu2} has a multi-residue at the point
$$s_n=\frac{1}{r},\qquad r(s_i-s_{i+1})=1.$$ From this we deduce that the Eisenstein series $E(h,s_1,\ldots,s_n)$ has a multi-residue at that point, and we denote the residue representation by $\Theta_{2n}^{(r)}$.  If 
 $\mu_0$ denotes the character of $T_{2n}$ defined by
\begin{equation}\label{mu3}
\mu_0(\text{diag}(a_1,\ldots,a_n,a_n^{-1},\ldots, a_1^{-1}))=|a_1|^{\frac{n}{r}}\cdots |a_{n-1}|^{\frac{2}{r}}|a_n|^{\frac{1}{r}} ,\notag
\end{equation}
then it follows that $\Theta_{2n}^{(r)}$ is a subquotient of the induced representation $\Ind_{B_{2n}^{(r)}({\A})}^{Sp_{2n}^{(r)}({\A})}\mu\delta_{B_{2n}}^{1/2}$. From this we deduce that $\Theta_{2n}^{(r)}$ is also a subrepresentation of the induced representation
$\Ind_{B_{2n}^{(r)}({\A})}^{Sp_{2n}^{(r)}({\A})}\chi_{Sp_{2n}^{(r)},\Theta}$ where
\begin{equation*}\label{mu4}
\chi_{Sp_{2n}^{(r)},\Theta}(\text{diag} (a_1,\ldots,a_n,a_n^{-1},\ldots, a_1^{-1}))=|a_1|^{\frac{n(r-1)}{r}}\cdots |a_{n-1}|^{\frac{2(r-1)}{r}}|a_n|^{\frac{r-1}{r}}.
\end{equation*}

We turn to even degree coverings.  We shall only be concerned with the covers of degree twice an odd integer.
If $r$ is odd, the definition of the theta representation for $Sp_{2m}^{(2r)}$ is similar. 
There is a small difference since the maximal parabolic subgroup of $Sp_{2m}$ whose Levi part is $GL_m$ splits under the double cover. Because of this, when computing the intertwining operator corresponding to the long Weyl element, 
one finds that its poles are determined by
\begin{equation}\label{mu5}
\frac{\prod_{i<j} [ \zeta_S(r(s_i-s_j))\zeta_S(r(s_i+s_j))] \prod_i\zeta_S(2rs_i)}{\prod_{i<j} [ \zeta_S(r(s_i-s_j)+1)\zeta_S(r(s_i+s_j)+1)] \prod_i\zeta_S(2rs_i+1)}.\notag
\end{equation}
Accordingly we define  $\Theta_{2m}^{(2r)}$ to be the multi-residue of the Eisenstein series at the point
$$s_m=\frac{1}{2r},\qquad r(s_i-s_{i+1})=1.$$
Then the representation $\Theta_{2m}^{(2r)}$ is a subquotient of the 
 induced representation $\Ind_{B_{2m}^{(r)}({\A})}^{Sp_{2m}^{(2r)}({\A})}\mu_e\delta_{B_{2m}}^{1/2}$  where $\mu_e$ denotes the character of $T_{2m}$ defined by
\begin{equation}\label{mu6}
\mu_e(\text{diag}(a_1,\ldots,a_m,a_m^{-1},\ldots, a_1^{-1}))=|a_1|^{\frac{2m-1}{2r}}\cdots |a_{m-1}|^{\frac{3}{2r}}|a_m|^{\frac{1}{2r}}.\notag
\end{equation}
It follows that $\Theta_{2m}^{(2r)}$ is also a subrepresentation of the induced representation $\Ind_{B_{2m}^{(r)}({\A})}^{Sp_{2m}^{(2r)}({\A})}\chi_{Sp_{2m}^{(2r)},\Theta}$ where
\begin{equation*} 
\chi_{Sp_{2n}^{(2r)},\Theta}(\text{diag} (a_1,\ldots,a_m,a_m^{-1},\ldots, a_1^{-1}))=|a_1|^{\frac{2mr-(2m-1)}{2r}}\cdots |a_{m-2}|^{\frac{6r-5}{2r}} |a_{m-1}|^{\frac{4r-3}{2r}}|a_m|^{\frac{2r-1}{2r}}.
\end{equation*}

We now develop some properties of the theta representations that follow from induction in stages, or, equivalently,
by taking an $(n-1)$-fold residue of $E(h,s_1,\dots,s_n)$ to obtain a maximal parabolic Eisenstein series attached to
theta representations of lower rank groups.
Let $1\le a\le n$. Denote by $P_{2n,a}$ the maximal parabolic subgroup of $Sp_{2n}$ whose Levi part is $GL_a\times Sp_{2(n-a)}$, and let $L_{2n,a}$ denote its unipotent radical.   We write $i$ for the inclusion of $GL_a$ in $P_{2n,a}$ in this paragraph
but  we suppress $i$ afterwards.
From the block compatibility of the cocycle $\sigma_{BLS}$ (Banks-Levy-Sepanski \cite{B-L-S}) and a short computation it follows
that if $g_1,g_2\in GL_a$ over a local field then
$\sigma(i(g_1),i(g_2))=\sigma_{GL_a}(g_1,g_2)^2 (\det g_1,\det g_2)^{-1}$, where $\sigma_{GL_a}$ is the metaplectic 2-cocycle 
of \cite{B-L-S} for $GL_a$.
Thus for both the $r$ and $2r$-fold covers, $i$ determines an $r$-fold cover of $GL_a$.  We write this cover $GL_a^{(r)}$ (which case we are in will always be clear from context).

Suppose first that $r$ is odd.
Let $\Theta_{GL_a}^{(r)}$ denote the theta representation of the group $GL_a^{(r)}({\A})$, as constructed in \cite{K-P}
(or the corresponding character if $a=1$).
Then it follows from induction in stages that 
$\Ind_{B_{2n}^{(r)}({\A})}^{Sp_{2n}^{(r)}({\A})}\mu_0\delta_{B_{2n}}^{1/2}$ is equal to $\Ind_{P_{2n,a}^{(r)}({\A})}^{Sp_{2n}^{(r)}({\A})}(\Theta_{GL_a}^{(r)}\otimes  \Theta_{2(n-a)}^{(r)}) \delta_{P_{2n,a}}^{\frac{r+1}{2r}}$.
From this we deduce the following. Let $E_a(g,s)$ denote the Eisenstein series of $Sp_{2n}^{(r)}({\A})$ associated with the induced representation $\Ind_{P_{2n,a}^{(r)}({\A})}^{Sp_{2n}^{(r)}({\A})}(\Theta_{GL_a}^{(r)}\otimes  \Theta_{2(n-a)}^{(r)}) \delta_{P_{2n,a}}^{s}$. 
(For the construction of such Eisenstein series more generally, see Brubaker and Friedberg \cite{B-F}.)
Then the representation
$\Theta_{2n}^{(r)}$ is a residue of $E_a(g,s)$ at the point $s=(r+1)/2r$. Of course, this can also be verified directly by studying the corresponding intertwining operators.

In the case of a cover of degree $2r$, $r$ odd, the situation is roughly similar.
In this case, the representation $\Ind_{B_{2m}^{(r)}({\A})}^{Sp_{2m}^{(2r)}({\A})}\mu_e\delta_B^{1/2}$ is equal to $\Ind_{P_{2m,a}^{(r)}({\A})}^{Sp_{2m}^{(2r)}({\A})}(\Theta_{GL_a}^{(r)}\otimes  \Theta_{2(m-a)}^{(2r)}) \delta_{P_{2n,a}}^{\frac{(r+1)(2m-a)+r} {2r(2m-a+1)}}$. Hence if we let $E_a(g,s)$ denote the Eisenstein series associated with the induced representation $\Ind_{P_{2m,a}^{(r)}({\A})}^{Sp_{2m}^{(2r)}({\A})}(\Theta_{GL_a}^{(r)}\otimes  \Theta_{2(m-a)}^{(2r)}) \delta_{P_{2n,a}}^{s}$, then we deduce that $\Theta_{2m}^{(2r)}$ is the residue of $E_a(g,s)$ at the point $s=\frac{(r+1)(2m-a)+r} {2r(2m-a+1)}$.

From these observations, we will deduce the following proposition.  Here and below, 
matrices are embedded in metaplectic groups by the trivial
section without additional notation, and we call metaplectic elements diagonal or unipotent when their projections to the 
linear group are.
 \begin{proposition}\label{propconstant}  Suppose that $r$ is odd.
Let $\theta_{2n}^{(r)}$ be in the space of $\Theta_{2n}^{(r)}$.  Then there exist functions $\theta_{GL_a}^{(r)}\in \Theta_{GL_a}^{(r)}$, 
$\theta_{2(n-a)}^{(r)}\in \Theta_{2(n-a)}^{(r)}$
such that for all diagonal $g$ in $GL_a^{(r)}({\A})$ which lies in the center of the Levi part of the parabolic group $P_{2n,a}^{(r)}({\A})$ and for all unipotent $h\in Sp^{(r)}_{2(n-a)}({\A})$, $v\in GL_a^{(r)}({\A})$, one has
\begin{equation}\label{constant10}
\int\limits_{L_{2n,a}(F)\backslash L_{2n,a}({\A})}\theta_{2n}^{(r)}(u(gv,h))du=\chi_{Sp_{2n}^{(r)} ,\Theta}(g)\,\theta_{GL_a}^{(r)}(v)\,
\theta_{2(n-a)}^{(r)}(h).
\end{equation}
A similar identity holds in the even-degree cover case.
\end{proposition}

The requirement that $h,v$ be unipotent could be dropped at the expense of a cocyle; being unipotent guarantees that it is $1$.

\begin{proof}
Since $\Theta_{2n}^{(r)}$ is the residue of the Eisenstein series $E_a(\cdot,s)$, we first consider the constant term
\begin{equation}\label{constant11}
\int\limits_{L_{2n,a}(F)\backslash L_{2n,a}({\A})}E_a(u(gv,h),s)\,du.
\end{equation}
For $\text{Re}(s)$ large we can unfold the Eisenstein series and deduce that \eqref{constant11} is equal to
\begin{equation}\label{constant12}
\sum_{w\in P_{2n,a}(F)\backslash Sp_{2n}(F)/P_{2n,a}(F)}\int\limits_{L_{2n,a}^w(F)
\backslash L_{2n,a}({\A})}\sum_{\gamma\in (w^{-1}P_{2n,a}(F)w\cap P_{2n,a}(F))\backslash P_{2n,a}(F)}f_s(w\gamma u(gv,h))\,du.
\end{equation}
Here $L_{2n,a}^w=w^{-1}L_{2n,a}w\cap L_{2n,a}^-$ where $L_{2n,a}^-$ is the conjugate of $L_{2n,a}$ by the long Weyl element. Also, we have $f_s\in
\Ind_{P_{2n,a}^{(r)}({\A})}^{Sp_{2n}^{(r)}({\A})}(\Theta_{GL_a}^{(r)}\otimes  \Theta_{2(n-a)}^{(r)}) \delta_{P_{2n,a}}^{s}$.
Notice that all elements in $P_{2n,a}(F)\backslash Sp_{2n}(F)/P_{2n,a}(F)$ can be chosen to be Weyl elements. 
Similarly to \cite{K-R}, Section 1.2, one can check that for all Weyl elements in $P_{2n,a}(F)\backslash Sp_{2n}(F)/P_{2n,a}(F)$ 
which are not the long Weyl element, the inner summation is a certain Eisenstein series or product 
of such series. Moreover, one can also check that the Eisenstein series in 
the corresponding summand in \eqref{constant12} is holomorphic at $s=(r+1)/2r$. Hence, taking the residue in \eqref{constant11} and \eqref{constant12} at $s=(r+1)/2r$, we are left only with the long
Weyl element, and for that element  we obtain the identity
$$\int\limits_{L_{2n,a}(F)\backslash L_{2n,a}({\A})}\theta_{2n}^{(r)}(u(gv,h))du=
\text{Res}_{s=(r+1)/2r}M_wf_s(gv,h)$$
where $M_w$ is the intertwining operator attached to $w$.
From this identity \eqref{constant10} follows.
\end{proof}

Similar statements for other groups are established in \cite{B-F-G} Proposition 3.4
and  \cite{F-G} Proposition 1.

We end this section with a general conjecture about the unipotent orbit attached to the representation $\Theta_{2n}^{(r)}$,
and give a consequence of this conjecture.
Recall that if $\pi$ denotes an automorphic representation of a reductive group,
then the set ${\mathcal O}(\pi)$ was defined in \cite{G1}.   It is the largest unipotent orbit that supports a nonzero coefficient
for this representation.
The extension of the definition of this set to the metaplectic groups is
clear.
In this paper we are interested in the set ${\mathcal O}(\Theta_{2n}^{(r)})$. The conjecture regarding this set is given as follows. Assume that $r<2n$. Write $2n=a(n,r)r+b(n,r)$ where $a(n,r)$ and $b(n,r)$ are both nonnegative numbers such that $0\le b(n,r)\le r-1$. 
We recall that a partition is a symplectic partition if any odd number in the partition occurs with even multiplicity. As defined in \cite{C-M}, given a partition $\lambda$ of $2n$, we define the $Sp$ collapse of $\lambda$ to be the greatest symplectic partition which is smaller than $\lambda$. We have

\begin{conjecture}\label{conj1}
Let $r$ be an odd number.
If $r<2n$, then the set ${\mathcal O}(\Theta_{2n}^{(r)})$ consists of the partition which is the $Sp$ collapse of the partition $(r^{a(n,r)}b(n,r))$.   If $r>2n$, then the representation $\Theta_{2n}^{(r)}$ is generic.
\end{conjecture}

The conjecture is studied in \cite{S1}.  We shall write Conjecture $1_{r,n}$ for this statement for a specific pair $(r,n)$.
We remark that an extension to even degree covers has not been formulated at this time.

To give a consequence of this conjecture, let $r$ be odd, let $a$ be an integer with $1\leq a\leq n-(r-1)/2$,
and let $b=n-(r-1)/2-a$. 
Let $U_{2n,n-b}$ be the unipotent radical of the standard parabolic subgroup of $Sp_{2n}$ whose Levi part is $GL_1^{n-b}\times Sp_{2b}$.  (In particular, $U_{2n,n}$ is the standard maximal unipotent subgroup of $Sp_{2n}$.)
Let $U_{2n,n-b,1}$ be the subgroup of $U_{2n,n-b}$ which consists of all matrices $u=(u_{i,j})\in U_{2n,n-b}$ such that $u_{n-b,i}=0$ for all $n-b+1\le i\le n$. Let $\underline{\alpha}=(\alpha_1,\ldots,\alpha_{a-1})$ where for all $i$ we have $\alpha_i\in\{0,1\}$.  Let $\psi$ be a nontrivial character of $F\backslash {\A}$,
which will be fixed throughout this paper.
Let
$\psi_{U_{2n,n-b,1},\underline{\alpha}}$ be the character of $U_{2n,n-b,1}$ given by
$$\psi_{U_{2n,n-b,1},\underline{\alpha}}(u)=\psi \left(\sum_{i=1}^{a-1} \alpha_i u_{i,i+1}+\sum_{j=a}^{n-b} u_{j,j+1}\right).$$
Then we have the following result, which will be used later.

\begin{lemma}\label{lemconstant1}  Suppose $r$ is odd, $r<2n$, and Conjecture~\ref{conj1}$_{r,n-i_0}$ holds for all $0\leq
i_0\leq a-1$.  Then
the integral
\begin{equation}\label{desc6}
\int\limits_{U_{2n,n-b,1}(F)\backslash U_{2n,n-b,1}({\A}) }
\theta_{2n}^{(r)}(ug)\,\psi_{U_{2n,n-b,1},\underline
{\alpha}}(u)\,du
\end{equation}
is zero for all choices of data, that is, for all $\theta_{2n}^{(r)}\in\Theta_{2n}^{(r)}$.
\end{lemma}

\begin{proof}
Consider first the case where $\alpha_i=1$ for all $i$. Then the Fourier coefficient \eqref{desc6} is the Fourier coefficient which corresponds to the unipotent orbit $((2n-2b)1^{2b})$. From Conjecture~\ref{conj1}$_{r,n}$ we 
have that ${\mathcal O}(\Theta_{2n}^{(r)})$ consists of the partition which is the $Sp$ collapse of the partition $(r^{a(n,r)}b(n,r))$. Hence $((2n-2b)1^{2b})$ is greater than or not related to ${\mathcal O}(\Theta_{2n}^{(r)})$. Indeed, this follows from the relation $2b=2n-2a-r+1$ which implies that $((2n-2b)1^{2b})=((r+2a-1)1^{2b})$. Thus, if $\alpha_i=1$ for all $i$, then the integral \eqref{desc6} is zero for all choices of data.

Next assume that at least one of the scalars $\alpha_i$ is zero. Let $i_0\le a-1$ be the largest index such that $\alpha_{i_0}=0$. Then the integral \eqref{desc6} is equal to
\begin{multline}\label{desc7}
\int\limits_{H(F)\backslash H({\A})}
\int\limits_{L_{2n,i_0}(F)\backslash L_{2n,i_0} ({\A})}
\int\limits_{U_{2(n-i_0),n-i_0-b,1}(F)\backslash U_{2(n-i_0),n-i_0-b,1}({\A}) }\\
\theta_{2n}^{(r)}(luhg)\,\psi_{U_{2(n-i_0) ,n-i_0-b,1}}(u)\,\psi_{H,\underline{\alpha}}(h)\,du\,dl\,dh.
\end{multline}
Here $H$ is a certain unipotent subgroup of $GL_{i_0}$ which will not be important to us.
Notice that the integration over $L_{2n,i_0}(F)\backslash L_{2n,i_0} ({\A})$ is the constant term along this unipotent group. Therefore, it follows from Proposition~\ref{propconstant} above
that the integration along the quotient space $U_{2(n-i_0),n-i_0-b,1}(F)\backslash U_{2(n-i_0),n-i_0-b,1}({\A})$ is the coefficient of the representation $\Theta_{2(n-i_0)}^{(r)}$ corresponding to the unipotent orbit $((2(n-i_0-b))1^{2b})$. Thus the vanishing of the integral \eqref{desc7} will follow if this unipotent orbit is greater than or not related to the 
orbit ${\mathcal O}(\Theta_{2(n-i_0)}^{(r)})$. From Conjecture~\ref{conj1}$_{r,n-i_0}$, this is the case if  $2(n-i_0-b)>r$. Since this 
inequality follows from the relations $2(n-a-b)=r-1$ and $a>i_0$, the result is proved.
\end{proof}

\section{The Descent Construction}\label{descent}
Recall that $U_{2m,k}$ denotes the unipotent radical of the standard parabolic subgroup of 
$Sp_{2m}$ whose Levi part  is $GL_1^k\times Sp_{2(m-k)}$. The quotient group $U_{2m,k-1}\backslash U_{2m,k}$ may be identified with the Heisenberg group in $2(m-k)+1$ variables, ${\mathcal H}_{2(m-k)+1}$. Indeed, there is a homomorphism $l : U_{2m,k}\mapsto {\mathcal H}_{2(m-k)+1}$ which is onto, and whose kernel is the group $U_{2m,k-1}$.

For the rest of this section we suppose that $r>1$ is odd. Set $r'=(r-1)/2$.
For $\alpha\in F^*$  let $\psi_{U,\alpha}$ be the character
of $U_{r-1,r'}(F)\backslash U_{r-1,r'}({\A})$ given by
$$\psi_{U,\alpha}(u)=\psi(u_{1,2}+u_{2,3}+\cdots +u_{r'-1,r'}+\alpha u_{r',r'+1})\qquad u=(u_{i,j}).$$ 

From now on we shall suppose that Conjecture~\ref{conj1}$_{r,r'}$ holds, that is, that
the representation $\Theta_{r-1}^{(r)}$ is generic.   This is known if $r=3$ \cite{K-P}.
Thus there is an $\alpha\in F^*$ and a choice of data, i.e.\ an automorphic function $\theta_{r-1}^{(r)}$ in the space of $\Theta_{r-1}^{(r)}$, such that the integral
\begin{equation}\label{wh1}
\int\limits_{U_{r-1,r'}(F)\backslash U_{r-1,r'}({\A})}\theta_{r-1}^{(r)}(ug) \,\psi_{U,\alpha}(u)\,du
\end{equation}
is not zero.
Replacing $\psi$ by $\psi^{\alpha^{-1}}$ and changing the function, we may assume that $\alpha=1$. In other words,
there is a choice of data $\theta_{r-1}^{(r)}$ such that the integral 
\begin{equation}\label{wh11}
\int\limits_{U_{r-1,r'}(F)\backslash U_{r-1,r'}({\A})}\theta_{r-1}^{(r)}(ug) \,\psi_{U}(u)\,du
\end{equation}
is not zero, where $\psi_U$ is defined as $\psi_{U,\alpha}$ with $\alpha=1$.

{\bf Remark:} {\sl In fact we expect that for all $\alpha\in F^*$, the integral \eqref{wh1} will be nonzero for some choice of data.}

Let $\Theta_{2m}^{\psi,\phi}$ denote the theta representation attached to the Weil representation and defined on the double cover of $Sp_{2m}({\A})$. Here $\phi$ is a Schwartz function on ${\A}^m$.  
Since $r$ is odd, using the isomorphism
$\mu_{2r}\cong \mu_2\times \mu_r$ we may map $g\in Sp_{2n-r+1}^{(2r)}({\A})$ to its image in the double and $r$-fold covers;
we will not introduce separate notation for this.
 To define the descent construction, consider the function of $g\in Sp_{2n-r+1}^{(2r)}({\A})$ given by
\begin{equation}\label{desc1}
f(g)=\int\limits_{U_{2n,r'}(F)\backslash U_{2n,r'}({\A}) }\theta_{2n-r+1}^{\psi,\phi}(l(u)g)\,\theta_{2n}^{(r)}
(ug) \,\psi_{U_{2n,r'}}(u)\,du.
\end{equation}
Here $\theta_{2n-r+1}^{\psi,\phi}$ and $\theta_{2n}^{(r)}$ are vectors in the spaces of the representations
$\Theta_{2n-r+1}^{\psi,\phi}$ and $\Theta_{2n}^{(r)}$ (resp.), $\psi_{U_{2n,r'}}$ is the character of $U_{2n,r'}({\A})$ given by $\psi_{U_{2n,r'}}(u_{i,j})=\psi(u_{1,2}+u_{2,3}+\cdots +u_{r'-1,r'})$, and $g\in Sp_{2n-r+1}$ is embedded in $Sp_{2n}$ as
$$g\mapsto \begin{pmatrix} I_{r'}&&\\ &g&\\ &&I_{r'}\end{pmatrix}.$$ 
(One could also consider descent integrals similar to \eqref{desc1} if $r$ is even but we shall not do so here.)
Then $f(g)$ is a genuine automorphic function defined on $Sp_{2n-r+1}^{(2r)}({\A})$.   
Let $\sigma_{2n-r+1}^{(2r)}$ denote the representation of  $Sp_{2n-r+1}^{(2r)}({\A})$ generated by all the functions $f(g)$.  

\section{Computation of the  Constant Term of the Descent Integral}
Let $V$ denote any standard unipotent subgroup of $Sp_{2n-r+1}$, and let $\psi_V$ 
be a character of $V(F)\backslash V({\A})$, possibly trivial.
In this paper we will compute integrals of the type
\begin{equation}\label{uni1}
\int\limits_{V(F)\backslash V({\A})}f(vg)\,\psi_V(v)\,dv,
\end{equation}
where $f(g)$ is given by \eqref{desc1}.

Using \eqref{desc1}, we arrive at an iterated integral.  We first unfold the theta function $\theta_{2n-r+1}^{\psi,\phi}$. Collapsing summation and integration, and then using the formulas for the action of the Weil representation
$\omega_\psi$, the integral \eqref{uni1} is equal to
\begin{equation}\label{desc2}
\int\limits_{{\A}^{n-r'}}\int\limits_{V(F)\backslash V({\A})}
\int\limits_{U_{2n,r',1}(F)\backslash U_{2n,r',1}({\A}) }
\omega_\psi(g)\phi(x)\,\theta_{2n}^{(r)}(uvj(x)g)\,
\psi_{U_{2n,r',1}}(u)\,\psi_V(v)\,du\,dv\,dx.
\end{equation}
Here $x=(x_1,\ldots, x_{n-r'})\in \A^{n-r'}$ is embedded in $Sp_{2n}$  via the map
 $j(x)=I_{2n}+x_1e^*_{r',r'+1}+x_2e^*_{r',r'+2}+\cdots +x_{n-r'}e^*_{r',n}$,
where $e_{i,j}$ denotes the $(i,j)$-th elementary matrix and $e^*_{i,j}=e_{i,j}-e_{2n-j+1,2n-i}$. 
Recall that the group $U_{2n,r',1}$ is the subgroup of $U_{2n,r'}$ consisting of all matrices  $u=(u_{i,j})\in U_{2n,r'}$ such that $u_{r',k}=0$ for all $r'+1\le k\le n$. The character $\psi_{U_{2n,r',1}}(u)$ is defined as the product of $\psi_{U_{2n,r'}}(u)$ and the character $\psi^0_{U_{2n,r',1}}$ of $U_{2n,r',1}$  given by $\psi^0_{U_{2n,r',1}}(u)=\psi(u_{r',2n-r'+1})$.

At this point we consider the case when $V=L_{2n-r+1,a}$ where $1\le a\le n-r'$ and $\psi_V$ is the trivial character. Thus, 
$V$ is the unipotent radical of the standard maximal parabolic subgroup of $Sp_{2n-r+1}$ whose 
Levi part is $GL_a\times Sp_{2n-2a-r+1}$. We recall that $2b=2n-2a-r+1$. 

Let $w_a$ denote the Weyl group element of $Sp_{2n}$ defined by
$$w_a=\begin{pmatrix} &I_a&&&\\ I_{n-a-b}&&&&\\ &&I_{2b}&&\\ &&&&I_{n-a-b}\\ &&&I_a& \end{pmatrix}.$$ Since $\theta_{2n}^{(r)}$ is left-invariant under rational points, we may conjugate by $w_a$ from left to right. 
Doing so, we see that the integral \eqref{desc2} is equal to
\begin{multline}\label{desc3}
\int\limits_{{\A}^{n-r'}}\int\limits_{Mat_{(n-a-b-1)\times a}(F)
\backslash Mat_{(n-a-b-1)\times a}({\A})}
\int\limits_{V_0(F)\backslash V_0({\A})}
\int\limits_{U_0(F)\backslash U_0({\A}) }\\
\omega_\psi(g)\phi(x)\,\theta_{2n}^{(r)}
(u_0v_0k(y)w_aj(x)g)\,\psi_{V_0}(v_0)\,du_0\,dv_o\,dy\,dx.
\end{multline}
The notations are defined as follows. Recall that $L_{2n,a}$ denotes the unipotent radical of the standard maximal parabolic subgroup of $Sp_{2n}$ whose Levi part is $GL_a\times Sp_{2(n-a)}$. The group $U_0$ is defined to be the subgroup of $L_{2n,a}$ consisting of all matrices in $Sp_{2n}$ of the form
$$\begin{pmatrix} I_a&0&*&*&*\\ &I_{n-a-b}&0&0&*\\
&&I_{2b}&0&*\\ &&&I_{n-a-b}&0\\ &&&&I_a\end{pmatrix}.$$
The embedding of $Mat_{(n-a-b-1)\times a}$ into $Sp_{2n}$ is given by
$$k(y)=\begin{pmatrix} I_a&&&&\\ y&I_{n-a-b}&&&\\ &&I_{2b}&&\\ &&&I_{n-a-b}&\\ &&&y^*&I_a\end{pmatrix},$$ 
where $y^*$ is chosen to make the matrix symplectic. Here the embedding of $y$ is such that it consists of all matrices of size $(n-a-b)\times a$ such that its bottom row is zero. Finally $V_0$ is equal to the group $U_{2(n-a),n-a-b,1}$, defined similarly to the definition of $U_{2n,a,1}$ reviewed after \eqref{desc2} (in particular, the group $V_0$ is a subgroup of $Sp_{2(n-a)}$), and 
the character $\psi_{V_0}:=\psi_{U_{2(n-a),n-a-b},1}$ as defined there. In coordinates, 
$\psi_{V_0}(v)=\psi(v_{a+1,a+2}+v_{a+2,a+3}+\cdots +v_{ n-b-1,n-b}+v_{n-b,n+b+1})$.

The next step is to carry out some root exchanges. The notion of {\sl root exchange} was described in detail in the authors' 
paper \cite{F-G}, Section 2.2; a more abstract formulation may be found in Ginzburg, Rallis and Soudry
 \cite{G-R-S3}, Section 7.1.
In our context, we perform the following root exchange.  For $z\in Mat_{a\times(n-a-b)}$, let $m(z)$ in $Sp_{2n}$ be given by
$$m(z)=\begin{pmatrix} I_a&z&&&\\ &I_{n-a-b}&&&\\ &&I_{2b}&&\\ &&&I_{n-a-b}&z^*\\ &&&&I_a\end{pmatrix}.$$  
We embed the group $Mat_{a\times (n-a-b-1)}(\A)$ inside $Sp_{2n}(\A)$ by considering all matrices $m(z)$ as above such that the first column of $z$ is zero. Expand the integral \eqref{desc3} along the group so obtained. Performing the root exchange with the group of all matrices $k(y)$ such that $y\in Mat_{(n-a-b-1)\times a}({\A})$, one sees that the integral \eqref{desc3} is equal to
\begin{multline}\label{desc4}
\int\limits_{{\A}^{n-r'}}\int\limits_{
Mat_{(n-a-b-1)\times a}({\A})}
\int\limits_{V_0(F)\backslash V_0({\A})}
\int\limits_{L^0_{2n,a}(F)\backslash L^0_{2n,a}({\A}) }
\\
\omega_\psi(g)\phi(x)\,\theta_{2n}^{(r)}(uv_0k(y)w_aj(x)g)\,
\psi_{V_0}(v_0)\,du\,dv_0\,dy\,dx.
\end{multline}
Here $L^0_{2n,a}$ is the subgroup of $L_{2n,a}$ which consists of all matrices $u=(u_{i,j})\in L_{2n,a}$ such that $u_{i,a+1}=0$ for all $1\le i\le a$.

Consider the quotient space of $L^0_{2n,a}\backslash L_{2n,a}$. 
This quotient may be naturally identified with the column vectors of size $a$. This group embeds in $Sp_{2n}$  as the group of all matrices of the form $I_{2n}+z_1e^*_{1,a+1}+z_2e^*_{2,a+1}+\cdots + z_ae^*_{a,a+1}$. Expand the integral \eqref{desc4} along this group. The group $GL_a(F)$  acts on this expansion with two orbits.

The contribution of the nontrivial orbit  to the integral \eqref{desc4} is the expression
\begin{multline}\label{desc5}
\int\limits_{{\A}^{n-r'}}\int\limits_{
Mat_{(n-a-b-1)\times a}({\A})}
\int\limits_{V_0(F)\backslash V_0({\A})}
\int\limits_{L_{2n,a}(F)\backslash L_{2n,a}({\A}) }
\\ \omega_\psi(g)\phi(x)\,\theta_{2n}^{(r)}
(uv_0k(y)w_aj(x)g)\,
\psi_{V_0}(v_0)\,\psi_{L_{2n,a}}(u)\,du\,dv_0\,dy\,dx,
\end{multline}
where for $u=(u_{i,j})\in L_{2n,a}(F)\backslash L_{2n,a}({\A})$,  $\psi_{L_{2n,a}}(u)=\psi(u_{a,a+1})$.
We claim that the integral \eqref{desc5} is zero for all choices of data. 

To prove this,  we carry out similar expansions repeatedly.  To start, consider the quotient space $L_{2n,a-1}\backslash L_{2n,a}$. This space may be identified with the column vectors of size $a-1$, embedded in $Sp_{2n}$ as the group of all matrices of the form $I_{2n}+z_1e^*_{1,a}+z_2e^*_{2,a}+\cdots + z_{a-1}e^*_{a-1,a}$. Expand the integral \eqref{desc5} along this group. The group $GL_{a-1}(F)$ acts on this expansion with two orbits. If $a-1>1$ we continue in this way.   Then the vanishing
of the integral \eqref{desc5} for all choices of data then follows from  Lemma~\ref{lemconstant1} above.

We conclude that the integral \eqref{uni1} with $V=U_{2n-r+1,a}$ and $\psi_V=1$ is equal to
\begin{multline}\label{desc8}
\int\limits_{{\A}^{n-r'}}\int\limits_{
Mat_{(n-a-b-1)\times a}({\A})}
\int\limits_{V_0(F)\backslash V_0({\A})}
\int\limits_{L_{2n,a}(F)\backslash L_{2n,a}({\A}) }
\\
\omega_\psi(g)\phi(x)\,\theta_{2n}^{(r)}
(uv_0k(y)w_aj(x)g) \,\psi_{V_0}(v_0)\,du\,dv_0\,dy\,dx.
\end{multline}

We shall prove momentarily that \eqref{desc8} is nonzero for some choice of data.  Supposing this for the moment,
let us analyze this integral when $g=\text{diag}(tI_a,I_{2n-r-2a+1},t^{-1}I_a)$ is an element in the center of the Levi part of the maximal parabolic subgroup of $Sp_{2n-r+1}$ whose Levi part is $GL_a\times Sp_{2n-r-2a+1}$, embedded in the $r$-fold metaplectic group via the trivial section. Suppose also that $t$ is an $r$-th power. Inserting $g$ into \eqref{desc8}, we
move it from right to left in the function $\theta_{2n}^{(r)}$.  Moving it past $j(x)$ and making the corresponding variable change in $x$ gives a factor of $|t|^{-a}$,
as well as a factor of $\gamma(t^a)|t|^{a/2}$ from the action of the Weil representation.  Then moving it past $k(y)$ we obtain a factor of $|t|^{-a(n-a-b-1)}$ from the corresponding change of variables in $y$. Clearly, $g$ commutes with $V_0$, and finally, applying Proposition \ref{propconstant} above we also get a factor of
$$\delta_{P_{2n,a}}^{\frac{r-1}{2r}}(\text{diag} (tI_a,I_{2(n-a)},t^{-1}I_a))=|t|^{a(r-1) (2n-a+1)/2r}$$
(here, recall that $P_{2n,a}$ is the maximal parabolic subgroup of $Sp_{2n}$ whose unipotent radical is $L_{2n,a}$).

We conclude that, for $g=\text{diag} (tI_a,I_{2(n-a)},t^{-1}I_a)$ as above and for $V=U_{2n-r+1,a}$, we have
\begin{equation*}\label{desc9}
\int\limits_{V(F)\backslash V({\A})}f(vg)\,dv=|t|^{\beta}\int\limits_{V(F)\backslash V({\A})}f(v)\,dv
\end{equation*}
where
$$\beta=\frac{a(r-1) (2n-a+1)}{2r}-\frac{a}{2}-a(n-a-b-1).$$
It follows from Section~\ref{def}  that $\chi_{Sp_{2n-r+1}^{(2r)},\Theta}(g)=|t|^\beta$. 

Finally, we prove that the integral \eqref{desc8} is not zero for some choice of data. This will
also imply that the representation $\sigma_{2n-r+1}^{(2r)}$ is not zero.
Suppose instead that \eqref{desc8} is zero for all choices of data. Arguing as in \cite{F-G}, Lemma 1, we may ignore the integration over the $x$ and $y$ variables, and deduce that the integral
\begin{equation}\label{desc10}
\int\limits_{V_0(F)\backslash V_0({\A})}
\int\limits_{L_{2n,a}(F)\backslash L_{2n,a}({\A}) \,}\theta_{2n}^{(r)}
(uv_0) \,\psi_{V_0}(v_0)\,du\, dv_0\notag
\end{equation}
is zero for all choices of data. From Proposition~\ref{propconstant} we deduce that the integral
\begin{equation}\label{desc11}
\int\limits_{V_0(F)\backslash V_0({\A})}
\theta_{2(n-a)}^{(r)}
(v_0h)\, \psi_{V_0}(v_0)\,dv_0
\end{equation}
is zero for all choice of data. Here $h\in Sp_{2(n-a)}^{(r)}({\A})$. From the definition of the group $V_0$ and the character $\psi_{V_0}$ it follows that the above Fourier coefficient corresponds to the unipotent orbit $((2(n-a-b))1^{2b})$. Since $2b=2n-2a-r+1$, this unipotent orbit is the same as $((r-1)1^{2n-2a-r+1})=((r-1)1^{2b})$. 

At this point we would like to use instances of Conjecture \ref{conj1}, but we need to be careful. The Fourier coefficients attached to the orbit $((r-1)1^{2b})$ are given by the integrals
\begin{equation}\label{desc12}
\int\limits_{V_0(F)\backslash V_0({\A})}
\theta_{2b+r-1}^{(r)}
(v_0h) \,\psi_{V_0,\alpha}(v_0)\,dv_0
\end{equation}
with $\alpha\in F^*$ and  $\psi_{V_0,\alpha}$ defined on $V_0\subset Sp_{2(n-a)}$ by
$$\psi_{V_0,\alpha}(v)=\psi(v_{1,2}+v_{2,3}+\cdots
+v_{r'-1,r'}+\alpha v_{r',2b+r'+1}).$$ 
Conjecture \ref{conj1} asserts that there is an $\alpha$ such that the integral \eqref{desc12} is not zero for some choice of data. However, we need to prove that the integral \eqref{desc11} is not zero for some choice of data, that is, the integral \eqref{desc12} with $\alpha=1$. This does not follow from Conjecture~\ref{conj1}.

To prove that integral \eqref{desc11} is not zero for some choice of data, we will use the non-vanishing of the integral \eqref{wh11}. Indeed, suppose instead that \eqref{desc11} is zero for all choices of data. Then the integral
\begin{equation}\label{desc13}
\int\limits_{L_{2b,b}(F)\backslash L_{2b,b}({\A})}
\int\limits_{V_0(F)\backslash V_0({\A})}
\theta_{2b+r-1}^{(r)}
(v_0l) \,\psi_{V_0}(v_0)\,dv_0\,dl
\end{equation}
is zero for all choices of data. We recall that $L_{2b,b}$ is the unipotent radical of 
the maximal parabolic subgroup of $Sp_{2b}$ whose Levi part is $GL_b$. Arguing as in the computation of  \eqref{desc2}, we obtain that the vanishing of \eqref{desc13} implies the vanishing of the integral
\begin{equation}\label{desc14}
\int\limits_{U_{r-1,r'}(F)\backslash U_{r-1,r'}({\A})}
\int\limits_{L_{2b+r-1,b}(F)\backslash L_{2b+r-1,b}({\A})}
\theta_{2b+r-1}^{(r)}(lv_0)\, \psi_{V_0}(v_0)\,dl\,dv_0.
\end{equation}
But the vanishing of  \eqref{desc14} for all choice of data implies that \eqref{wh11} is zero for all choice of data,
by Proposition~\ref{propconstant}. This is a contradiction.

Since the constant terms have been bounded and the descent is a compact integral of an automorphic form, the descent
and its derivatives have uniformly moderate growth.  Then from truncation as in the Corollary in \cite{M-W}, I.2.12, together with
the above calculations, one
obtains the following result. 

\begin{proposition}\label{l2}  Let $r$ be odd, $r<2n$, and suppose that Conjecture 1$_{r,i}$ holds for all $i$ with $r'\leq i\leq n$.  
Then the representation $\sigma_{2n-r+1}^{(2r)}$ is not zero and it is a sub-representation of $L^2(Sp_{2n-r+1}^{(2r)}(F)\backslash
Sp_{2n-r+1}^{(2r)}({\mathbb{A}}))$. Moreover its projection to the residual spectrum is
nonzero.
\end{proposition}
\noindent
 One could also deduce this following \cite{F-G}, Propositions 4 and 5, by a sequence
of expansions along abelian unipotent radicals.  We omit the details, as in fact
we expect that a stronger statement holds (compare \cite{F-G}, Conjecture 1).

\begin{conjecture}\label{conjecture-new-strong}
The representation $\sigma_{2n-r+1}^{(2r)}$ is equal to $\Theta_{2n-r+1}^{(2r)}$.
\end{conjecture}

\section{The Whittaker Coefficient of the Descent}

In this section we study those cases in which the descent representation $\sigma_{2n-r+1}^{(2r)}$ is generic. 
Let us explain how these cases are predicted by the 
dimension equation of the second author \cite{G2}.   It follows from this equation
together with \eqref{desc1} that the following equation should hold, where the dimension of a representation means
its Gelfand-Kirillov dimension:
\begin{equation}\label{dim1}
\text{dim}\ \Theta_{2n-r+1}^{\psi,\phi}\ + \text{dim}\ \Theta_{2n}^{(r)}=\text{dim}\ U_{2n,r'} +
\text{dim}\ \sigma_{2n-r+1}^{(2r)}.
\end{equation}
Since $\Theta_{2n-r+1}^{\psi,\phi}$ is the minimal representation of $Sp_{2n-r+1}^{(2)}({\A})$,
its dimension is $\frac{1}{2}(2n-r+1)$. An easy computation shows that $\text{dim}\ U_{2n,r'}=
\frac{1}{4}(r-1)(4n-r+1)$. If we want the representation $\sigma_{2n-r+1}^{(2r)}$ to be generic, 
its dimension should equal $\frac{1}{4}(2n-r+1)^2$. Using these values,  \eqref{dim1} gives
\begin{equation}\label{dim2}
\text{dim}\ \Theta_{2n}^{(r)}=n^2-n+\frac{1}{2}(r-1).
\end{equation}
It is not hard to check that for a representation $\Theta_{2n}^{(r)}$ which satisfies both 
equation \eqref{dim2} and Conjecture \ref{conj1}$_{r,n}$, we must have $n\le r <2n$. If $n=2k$, then 
for all $0\le i\le k-1$ the dimension of the unipotent orbit $((4k-2i-2)(2i+2))$ is equal to
$n^2-n+\frac{1}{2}(r-1)$ with $r=4k-2i-1$. If $n=2k+1$, then 
for all $0\le i\le k-1$ the dimension of the unipotent orbit $((4k-2i)(2i+2))$ is equal to
$n^2-n+\frac{1}{2}(r-1)$ with $r=4k-2i+1$. In this case we also have the unipotent orbit 
$((2k+1)^2)$ whose dimension is $n^2-n+\frac{1}{2}(r-1)$ with $r=2k+1$. The above argument motivates the following
result.
\begin{theorem}\label{th1}
Suppose that $r$ is odd, $n\le r <2n$. Suppose that Conjecture~\ref{conj1}$_{r,i}$ holds for $r'\leq i\leq n$.  If
$r\neq n, n+1$, suppose that Assumption 1 below holds as well.
Then the representation $\sigma_{2n-r+1}^{(2r)}$ is generic. Moreover, in the special case where $r=n$, the Whittaker coefficient of 
$\sigma_{n+1}^{(2n)}$ attached to factorizable data $\theta_{2n}^{(n)}\in \Theta_{2n}^{(n)}$ and $\phi$ is factorizable.
\end{theorem}

We shall see below that the hypotheses concerning Conjecture~\ref{conj1}$_{r,i}$ are satisfied when $r=n=3$. 

\begin{proof}
The proof is based on a computation of the integral \eqref{uni1}
with $V$ the maximal unipotent subgroup of $Sp_{2n-r+1}$ and $\psi_V$ the Whittaker  character of $V$. 

According to Conjecture \ref{conj1}$_{r,n}$ the unipotent orbit attached to the
representation $\Theta_{2n}^{(r)}$ is computed as follows. 
Let $r=2n-2l+1$ for convenience, and 
write $2n=a(n,r)(2n-2l+1)+b(n,r)$. Then $a(n,r)=1$ and $b(n,2)=2l-1$. Thus, the unipotent orbit 
attached to the representation $\Theta_{2n}^{(r)}$ is the $Sp$ collapse of  
$((2n-2l+1)(2l-1))$. This is $((2n-2l)(2l))$ unless $n=2l-1$, in which case it is $((2l-1)^2)$.  This latter is the case $n=r$.

We consider first the case where the unipotent orbit attached to the representation $\Theta_{2n}^{(r)}$
is of the form $((2n-2l)(2l))$. (In the above notations we have $l=i+1$.) 
We first describe the set of Fourier coefficients attached to this unipotent orbit.
Let $P'_l$ denote the standard parabolic subgroup of $Sp_{2n}$ whose Levi part is $GL_1^{n-2l}\times GL_2^l$. We embed this group in $Sp_{2n}$ as the group of all matrices of the form $\text{diag}(a_1,
\ldots, a_{n-2l},g_1,\ldots, g_l, g_l^*,\ldots, g_1^*,a_{2n-l}^{-1},\ldots,a_1^{-1})$. Here $a_i\in GL_1$ and $g_j
\in GL_2$ and the starred entries are determined by the requirement that the matrix be in $Sp_{2n}$. Let $U'_l$ denote the 
unipotent radical of $P'_l$. Then the Fourier coefficients of an automorphic function $\varphi$ attached to the 
unipotent orbit  $((2n-2l)(2l))$ are given by
\begin{equation}\label{fc1}
\int\limits_{U'_l(F)\backslash U'_l({\A})}\varphi(ug)\,\psi_{U'_l,a,b}(u)\,du
\end{equation}
where  $a,b\in F^*$ and the character $\psi_{U'_l,a,b}$ is defined by 
\begin{equation}\label{psi1}
\psi_{U'_l,a,b}(u)=\psi\left(\sum_{i=1}^{n-2l} u_{i,i+1} + \sum_{j=n-2l+1}^{n-2} u_{j,j+2}+au_{n-1,n+2}+bu_{n,n+1}\right).
\end{equation}

Thus, to assert that the unipotent orbit $((2n-2l)(2l))$ is attached to a certain 
representation, means first that there is a choice of $a$ and $b$ such that the Fourier 
coefficient \eqref{fc1} is not zero for some choice of data, and second that for all orbits 
which are greater than or not related to $((2n-2l)(2l))$ all corresponding 
Fourier coefficients are zero. However, Conjecture~\ref{conj1}$_{r,n}$ does not specify for which values
of $a$ and $b$ the integral \eqref{fc1} is nonzero for some choice of data. 
To prove the Theorem, we will need a compatibility assumption which is stated as follows. 
\begin{assumption}\label{assume1}
Suppose that $r\neq n, n+1$.  Then there is a choice of vector $\varphi$ in the space of $\Theta_{2n}^{(r)}$
such that the integral \eqref{fc1} with $a=1$  is not zero.  
\end{assumption}
This assumption, which will not be needed when the orbit is of the form $(n^2)$,
is used to guarantee the non-vanishing of the
Whittaker coefficient of the descent. 

At this point we 
begin the computation of the Whittaker coefficient when $n\neq r$; the case $n=r$ will be treated separately. We compute the integral \eqref{uni1} with $V=U_{2n-r+1,n-r'}$, that is,
\begin{equation}\label{wh3}
\int\limits_{U_{2n-r+1,n-r'}(F)\backslash U_{2n-r+1,n-r'}({\A})}f(ug)\,\psi_{U,b}(u)\,du.
\end{equation}
As above $r'=(r-1)/2$ and the character $\psi_{U,b}$, $b\in F^*$, is given by
\begin{equation*}\label{psi2}
\psi_{U,b}=\psi(u_{1,2}+u_{2,3}+\cdots +u_{n-r'-1,n-r'}+bu_{n-r',n-r'+1}).
\end{equation*}
Similarly to the analysis of  \eqref{uni1} and \eqref{desc2}, the integral \eqref{wh3}
is equal to
\begin{multline}\label{wh4}
\int\limits_{{\A}^{n-r'}}\int\limits_{U_{2n-r+1,n-r'}(F)\backslash U_{2n-r+1,n-r'}({\A})}
\int\limits_{U_{2n,r',1}(F)\backslash U_{2n,r',1}({\A}) }\\
\omega_\psi(g)\phi(x)\,\theta_{2n}^{(r)}(uvj(x)g)\,
\psi_{U_{2n,r',1}}(u)\,\psi_{U,b}(v)\,du\,dv\,dx.
\end{multline}

Define the following Weyl element $w_0$ of $Sp_{2n}$. First, let $w$ denote the Weyl element of
$GL_n$ given by $w=\begin{pmatrix} I_{n-2l+1}&\\ &\nu\end{pmatrix}$ with $l=n-r'$ and $\nu=(\nu_{i,j})\in GL_{2l-1}$ defined
as follows. Let $\nu_{i,j}=1$ for the pairs 
$(i,j)=(1,l); (3,l+1); (5,l+2);\ldots ;(2l-1,2l-1)$ and $(i,j)=(2,1); (4,2); (6,3)\ldots (2l-2, l-1)$ and $\nu_{i,j}=0$ otherwise.
Then let $w_0=\begin{pmatrix} w&\\ &w^*\end{pmatrix}$ where the star indicates that the matrix is in $Sp_{2n}$.  

Since $\theta_{2n}^{(r)}$ is left-invariant under $w_0$, inserting $w_0$ into $\theta_{2n}^{(r)}$ and moving it rightward, 
the integral \eqref{wh4} is equal to
\begin{equation}\label{wh5}
\int\limits_{{\A}^{n-r'}}\int\limits_{V_1(F)\backslash V_1({\A})}
\int\limits_{U'_{l,1}(F)\backslash U'_{l,1}({\A}) }
\omega_\psi(g)\phi(x)\,\theta_{2n}^{(r)}(uvw_0j(x)g)\,
\psi_{U'_l,1,b}(u)\,du\,dv\,dx.
\end{equation}
Here $U'_{l,1}$ is the subgroup consisting of $u=(u_{i,j})\in U'_l$ such that
$u_{2m,2k+1}=0$ for all $1\le m\le l-1$ and $m\le k\le l-1$, and the character $\psi_{U'_l,1,b}$ is the
character $\psi_{U'_l,a,b}$ defined in \eqref{psi1} with $a=1$. Also, the group $V_1$ is the unipotent
subgroup of $Sp_{2n}$ consisting of all matrices of the form $I_{2n}+\sum_{i,j}r_{i,j}e_{i,j}^*$ where
the sum is over all pairs $(i,j)=(2k+1,2m)$ with $1\le m\le l-1$ and $m-1\le k\le l-2$. 

Let $Z_l$ denote the unipotent subgroup of $U'_l$ consisting of all matrices of the form 
$I_{2n}+\sum_{i,j}r_{i,j}e_{i,j}^*$ where the sum is over all pairs $(i,j)=(2m,2k+1)$ with
$1\le m\le l-1$ and $m\le k\le l-1$.
Then $Z_l$ is a subgroup of $U'_l$ which satisfies the condition
that $U'_l=Z_l\cdot U'_{l,1}$. We perform a root exchange in the integral \eqref{wh5}. 
Expanding the integral \eqref{wh5} along the quotient $Z_l(F)\backslash Z_l({\A})$
and exchanging roots with the group $V_1$, we deduce that this integral is equal to
\begin{equation}\label{wh6}
\int\limits_{{\A}^{n-r'}}\int\limits_{V_1({\A})}
\int\limits_{U'_l(F)\backslash U'_l({\A}) }
\omega_\psi(g)\phi(x)\,\theta_{2n}^{(r)}(uvw_0j(x)g)\,
\psi_{U'_l,1,b}(u)\,du\,dv\,dx.
\end{equation}
Using \cite{F-G}, Lemma 1, we deduce that the integral \eqref{wh6} is not zero for some choice of data if and only if the integral \eqref{fc1} with $a=1$ is not zero for some choice of data.
When $r\neq n+1$, i.e.\ $n\ne 2l$, we use Assumption~\ref{assume1} to deduce the nonvanishing of integral \eqref{wh6}, or equivalently,
the nonvanishing of a Whittaker coefficient of $\sigma_{2n-r+1}^{(r)}$. 

When $r=n+1$, so $n=2l$, we may deduce the nonvanishing of the integral 
\eqref{wh6} without requiring an analogue of Assumption~\ref{assume1}. 
Indeed, this situation is essentially the same as the cases studied in Ginzburg, Rallis and Soudry \cite{G-R-S1, G-R-S2}. 
In those works, the residue representation is defined on the linear group
$Sp_{2n}({\A})$ rather than a covering group. In our case, we are concerned with the $r=(n+1)$-fold cover, and the representation we use to define the descent is $\Theta_{2n} ^{(n+1)}$. However, the crucial ingredient for the computations of the Whittaker coefficients is the unipotent orbit attached to the residue representations. In both cases it is the orbit  $(n^2)$. Since the computations for the proof of the Theorem involve only manipulations of  unipotent groups, the cover does not enter. 
Since the notations in those references are different,  for the convenience of the reader we sketch the argument. 

Starting with the integral \eqref{wh6}, we choose $b=-1/4$. Let $\gamma=\begin{pmatrix}
1& \\ -1&1\end{pmatrix}\begin{pmatrix} 1& \tfrac12\\ &1\end{pmatrix}$, and define
$\gamma_0=\text{diag} (\gamma,\ldots,\gamma,\gamma^*,\ldots,\gamma^*)\in Sp_{2n}(F)$. This is
analogous to the matrix $a$ defined in \cite{G-R-S1}, (4.8). 
Since $\theta_{2n}^{(n+1)}$ is invariant under $\gamma_0$, inserting it in the integral \eqref{wh6} and moving
it rightward we obtain 
\begin{equation}\label{wh7}
\int\limits_{{\A}^{n-r'}}\int\limits_{V_1({\A})}
\int\limits_{U'_l(F)\backslash U'_l({\A}) }
\omega_\psi(g)\phi(x)\,\theta_{2n}^{(n+1)}(u\gamma_0vw_0j(x)g)\,
\psi_{U'_l}(u)\,du\,dv\,dx
\end{equation}
where $\psi_{U'_l}$ is given by
$$\psi_{U'_l}(u)=\psi(u_{1,3}+u_{2,4}+u_{3,5}+u_{4,6}+\cdots +u_{n-3,n-1}+u_{n-2,n}+
u_{n-1,n+1}).$$ 

Define the following Weyl element $w_0'$ of $Sp_{2n}$. If $w_0'=(w'_{0,i,j})$ then
$w'_{0,i,2i-1}=1$ for all $1\le i\le n$. All other entries are $0,\pm 1$ and  are determined uniquely such that  $w_0'\in Sp_{2n}$. This Weyl element was denoted by $\nu$ in \cite{G-R-S1}, p.\ 881. Inserting $w_0'$ in the integral \eqref{wh7}, we obtain
\begin{equation}\label{wh8}
\int\limits_{{\A}^{n-r'}}\int\limits_{V_1({\A})}
\int\limits_{Y(F)\backslash Y({\A})}\int\limits_{U_{2n,n,0}(F)\backslash U_{2n,n,0}({\A})}
\omega_\psi(g)\phi(x)\,\theta_{2n}^{(n+1)}(uyw_0'\gamma_0vw_0j(x)g)\,
\psi_{U_{2n,n}}(u)\,du\,dy\,dv\,dx.
\end{equation}
Here $Y$ is the unipotent group consisting of all matrices in $Sp_{2n}$ of the form $\begin{pmatrix}
I_n& \\ y&I_n\end{pmatrix}$ such that $y_{i,j}=0$ if $i\ge j$. The group $U_{2n,n,0}$ is the
subgroup of $U_{2n,n}$ such that $u_{i,j}=0$ for all $1\le i\le n$ and $n+1\le j\le n+i$. Finally,
the character $\psi_{U_{2n,n}}$ is the character of $U_{2n,n}$ defined by $\psi_{U_{2n,n}}(u)=\psi(u_{1,2}+u_{2,3}+\cdots u_{n-1,n})$.

At this point we carry out a sequence of root exchanges. This process is described in detail in \cite{G-R-S1}
Section 5 so we will omit the details here. Carrying out this process, we deduce that the integral
\eqref{wh6} is equal to
\begin{equation}\label{wh9}
\int\limits_{{\A}^{n-r'}}\int\limits_{V_1({\A})}
\int\limits_{ Y({\A})}\int\limits_{U_{2n,n}(F)\backslash U_{2n,n}({\A})}
\omega_\psi(g)\phi(x)\,\theta_{2n}^{(n+1)}(uyw_0'\gamma_0vw_0j(x)g)\,
\psi_{U_{2n,n}}(u)\,du\,dy\,dv\,dx.
\end{equation}
Then once again \cite{F-G}, Lemma 1, implies that integral 
\eqref{wh9} is not zero for some choice of data if and only if the inner integration along
the quotient $U_{2n,n}(F)\backslash U_{2n,n}({\A})$ is not zero for some choice of data. The
nonvanishing of this integral follows from Proposition~\ref{propconstant}
with $a=n$ and the genericity of the representation $\Theta_{GL_n}^{(n+1)}$ \cite{K-P}.

To complete the proof of the Theorem we need to consider the case when $n$ is odd and $r=n$. In
this case we have $l=(n+1)/2$.  To use formulas from the previous case, we change the notation, and now
let $P'_l$ be the standard parabolic subgroup of $Sp_{2n}$ whose Levi part is $GL_1\times GL_2^{{(n-1)}/{2}}$ and $U'_l$ be its unipotent radical.  (Above we would have written $P'_{l-1}$, $U'_{l-1}$.)
 We again
start with the integral \eqref{wh3}, define $w_0$ as above, and after conjugation we obtain the
integral \eqref{wh5}, with the groups $U'_{l,1}$ and $V_1$ changed as follows.
Now $U'_{l,1}$ is the subgroup of $U'_l$ consisting of $u=(u_{i,j})\in U'_l$ such that $u_{2m-1,2k}=0$ for all $1\le m\le l-1$ 
and $m\le k\le l-1$, and $V_1$ is the unipotent subgroup of $Sp_{2n}$ consisting of all matrices of the form $I_{2n}+\sum_{i,j}r_{i,j}e_{i,j}^*$ where the sum ranges over all pairs $(i,j)=(2k,2m-1)$ where $1\le m\le l-1$ and $m-1\le k\le l-2$.

Let $Z_l$ now be the subgroup of $U'_l$ consisting of all matrices of the form
$I_{2n}+\sum_{i,j}r_{i,j}e_{i,j}^*$ where the sum is over all pairs $(i,j)=(2m-1,2k)$ where 
$1\le m \le l-1$ and $m\le k\le l-1$ together with the condition that $r_{1,2}=0$.  Notice that
in this case it is not true that $U'_l=Z_l\cdot U'_{l,1}$. Performing root exchange similarly to our treatment of the
integral \eqref{wh5} we obtain
\begin{equation}\label{wh10}
\int\limits_{{\A}^{n-r'}}\int\limits_{V_1({\A})}
\int\limits_{U'_{l,0}(F)\backslash U'_{l,0}({\A}) }
\omega_\psi(g)\phi(x)\,\theta_{2n}^{(n)}(uvw_0j(x)g)\,
\psi_{U'_l,1,b}(u)\,du\,dv\,dx.\notag
\end{equation}
Here, $U'_{l,0}$ is the subgroup of $U'_l$ consisting of all matrices $u=(u_{i,j})\in U'_l$
such that $u_{1,2}=0$, and the character $\psi_{U'_l,1,b}$ is given by
\begin{equation}\label{psi3}
\psi_{U'_l,1,b}(u)=\psi\left(\sum_{i=1}^{n-2} u_{i,i+2} + u_{n-1,n+2}+bu_{n,n+1}\right).\notag
\end{equation}

Expand the above integral along the group $z(y)=I_{2n}+ye_{1,2}^*$. We obtain
\begin{equation}\label{wh111}
\int\limits_{{\A}^{n-r'}}\int\limits_{V_1({\A})}
\sum_{\alpha\in F}\int\limits_{F\backslash {\A}}
\int\limits_{U'_{l,0}(F)\backslash U'_{l,0}({\A}) }
\omega_\psi(g)\phi(x)\,\theta_{2n}^{(n)}(z(y)uvw_0j(x)g)\,
\psi_{U'_l,1,b}(u)\,\psi(\alpha y)\,du\,dy\,dv\,dx.\notag
\end{equation}
Assume that $b=-1/4$. Then it is not hard to check that if $\alpha\ne -1/2$ then the two inner
integrations give a coefficient which is associated to the unipotent orbit $((n+1)(n-1))$.
Since $n$ is odd this is a symplectic partition. However, it follows from Conjecture~\ref{conj1}$_{n,n}$
that the unipotent orbit associated with $\Theta_{2n}^{(n)}$ is $(n^2)$. Therefore, the only
nonzero contribution to the above integral is from $\alpha =1/2$. Let $\gamma=\begin{pmatrix}
1& \\ 1&1\end{pmatrix}\begin{pmatrix} 1& -\tfrac12\\ &1\end{pmatrix}$ and let
$\gamma_0=\text{diag} (\gamma,\ldots,\gamma,\gamma^*,\ldots,\gamma^*)\in Sp_{2n}(F)$. Using the invariance of 
$\theta_{2n}^{(n)}$ by
this element and moving it rightward, we obtain the integral 
\begin{equation}\label{wh12}
\int\limits_{{\A}^{n-r'}}\int\limits_{V_1({\A})}
\int\limits_{U'_l(F)\backslash U'_l({\A}) }
\omega_\psi(g)\phi(x)\,\theta_{2n}^{(n)}(u\gamma vw_0j(x)g)\,
\psi_{U'_l}(u)\,du\,dv\,dx,
\end{equation}
where $\psi_{U'_l}$ is given by
\begin{equation}\label{psi4}
\psi_{U'_l}(u)=\psi(u_{1,2}+u_{2,4}+u_{3,5}+\cdots +u_{n-3,n-1}+u_{n-2,n}+
u_{n-1,n+1}).\notag
\end{equation}

Next we introduce the Weyl element $w_0^*\in Sp_{2n}$ defined by 
$$w_0^*=\begin{pmatrix} 1&&\\ &w_0'&\\ &&1\end{pmatrix}.$$ 
(Here  $w_0'$ was defined following \eqref{wh7}.) Using invariance by $w_0^*$ and moving this element rightward,
\eqref{wh12}
is equal to
\begin{equation}\label{wh13}
\int\limits_{{\A}^{n-r'}}\int\limits_{V_1({\A})}
\int\limits_{Y_0(F)\backslash Y_0({\A})}\int\limits_{U_{2n,n,0}(F)\backslash U_{2n,n,0}({\A})}
\omega_\psi(g)\phi(x)\,\theta_{2n}^{(n)}(uy_0w_0^*\gamma_0vw_0j(x)g)\,
\psi_{U_{2n,n}}(u)\,du\,dy_0\,dv\,dx.\notag
\end{equation}
Here $U_{2n,n,0}$ is now the subgroup of $U_{2n,n}$ of $u=(u_{i,j})\in U_{2n,n}$
such that $u_{i,j}=0$ for $2\le i\le n$ and $n+1\le j\le n+i-1$, and $Y_0$ is the group of all matrices of the form
$$y_0=\begin{pmatrix} 1&&\\ &y'&\\ &&1\end{pmatrix}$$ 
such that $y'\in Y$ (where $Y$, as well as the character $\psi_{U_{2n,n}}$, was defined following \eqref{wh8}).  

Carrying out root exchanges similarly to our treatment of \eqref{wh8}, we obtain the integral
\begin{equation}\label{wh14}
\int\limits_{{\A}^{n-r'}}\int\limits_{V_1({\A})}
\int\limits_{Y_0({\A})}\int\limits_{U_{2n,n}(F)\backslash U_{2n,n}({\A})}
\omega_\psi(g)\phi(x)\,\theta_{2n}^{(n)}(uy_0w_0^*\gamma_0vw_0j(x)g)\,
\psi_{U_{2n,n}}(u)\,du\,dy_0\,dv\,dx.
\end{equation}
It is not hard to check that the integral \eqref{wh14} is not zero for some choice of data.
To complete the proof of the Theorem, observe that Proposition~\ref{propconstant} 
allows us to express the integral of $\psi_{U_{2n,n}}$ times any right translate of $\theta_{2n}^{(n)}$ over
$U_{2n,n}(F)\backslash U_{2n,n}({\A})$ in terms of the Whittaker function of the representation $\Theta_{GL_n}^{(n)}$.
Since the Whittaker functional of this representation is unique up to scalars (see Kazhdan-Patterson \cite{K-P}, Theorem
II.2.1), we conclude that  the integral \eqref{wh14} is
factorizable provided the data $\theta_{2n}^{(n)}\in \Theta_{2n}^{(n)}$ and $\phi$ are factorizable.
\end{proof}

\section{Some Local Computations}

We have just seen that for factorizable data the Whittaker coefficient of
$\sigma_{n+1}^{(2n)}$, $n$ odd, can be expressed as a product of local Whittaker functionals. We now analyze the corresponding
local integral in the unramified case.  Consider a
place $\nu$ where all data is unramified.  In this Section, $F$ denotes the corresponding local field, and in
all integrals we understand that we are taking the $F$-valued points of the algebraic groups that are shown.
We continue to use the convention that all matrices are embedded in the corresponding metaplectic groups via the trivial section, unless otherwise noted.

In this section, let $T_{2n}^{(n)}$ denote the inverse image in the local $n$-fold metaplectic group of the 
standard maximal torus of $Sp_{2n}(F)$,
and let $Z(T_{2n}^{(n)})$ denote its center.
Let $p$ be a generator of the maximal ideal in the ring of integers ${\mathcal O}$ of $F$.
Since $n$ is odd, $-1$ is an $n$-th power
and hence $(p,p)=1$.  This implies that $(p^{k_1},p^{k_2})=1$ for all integers $k_1$, $k_2$.
The subgroup of $T_{2n}^{(n)}$ generated by $Z(T_{2n}^{(n)})$ and by all matrices
$\text{diag}(p^{k_1},\ldots,p^{k_n},p^{-k_n},\ldots,p^{-k_1})$ is a maximal abelian subgroup of
$T_{2n}^{(n)}$. Denote this subgroup by $T_{2n,0}^{(n)}$. For a diagonal matrix 
$t=\text{diag}(a,a^*)\in Sp_{2n}(F)$ we shall denote $t_0=a\in GL_n(F)$.   Also, recall that since we are in the unramified case the 
maximal compact subgroup $Sp_{2n}({\mathcal{O}})$ of $Sp_{2n}(F)$ splits under the $n$-fold cover; fix an
injection $\iota$
of $Sp_{2n}({\mathcal{O}})$ into  $Sp_{2n}^{(n)}(F)$.

Let $f_W^{(n)}$ be the function on $Sp_{2n}(F)$ defined as follows. First, it is right-invariant under 
$\iota(Sp_{2n}({\mathcal{O}}))$.
Second,  $f_W^{(n)}(ug)=\psi_{U_{2n,n}}(u)f_W^{(n)}(g)$ for all $u\in U_{2n,n}(F)$. (The character $\psi_{U_{2n,n}}$ was defined 
before \eqref{wh9} above.) Finally, for all $(t,\zeta)\in  T_{2n,0}^{(n)}$ ($t=\text{diag}(a,a^*)\in Sp_{2n}(F)$, $\zeta\in \mu_{n})$ we have 
$f_W^{(n)}((t,\zeta))=\zeta \,\delta_Q^{\frac{n-1}{2n}}(t)\,W_{GL_n}^{(n)}((t_0,1))$. Here $W_{GL_n}^{(n)}$
is the normalized Whittaker function for the local theta representation of $GL_{n}^{(n)}$ that is unramified with respect to the
subgroup $GL_n(\mathcal{O})$ embedded in  $GL_{n}^{(n)}(F)$ in a way that is compatible with $\iota$,
and $Q=P_{2n,n}$ is the maximal parabolic subgroup of $Sp_{2n}$ whose Levi part is $GL_n$. 
(Recall from \cite{K-P} that the space of 
Whittaker
functionals for the theta function of $GL_n^{(n)}(F)$ is one dimensional.)

Consider the integral 
\begin{equation}\label{unra1}
W_{Sp_{n+1}}^{(2n)}(g)=\int\limits_{F^{n-r'}}\int\limits_{V_1}
\int\limits_{Y_0}\omega_\psi(g)\phi(x)\,f_W^{(n)}
(y_0w_0^*\gamma_0vw_0j(x)g)\,dy_0\,dv\,dx,
\end{equation}
where the function $\phi$ is the normalized unramified Schwartz function.

It may be desirable to compute $W_{Sp_{n+1}}^{(2n)}(g)$ for an arbitrary torus element; however, 
a general formula would be quite complicated. We will concentrate on a special
case (though for $n=3$ it is general)
where the computations, though not trivial, are simpler. Let $n_1,n_2$ be non-negative integers,
$a=p^{n_1}$ and $b=p^{n_2}$, and let 
$$g=\text{diag}(ab,b,I_{n-3},b^{-1},a^{-1}b^{-1}).$$
Substituting this element in \eqref{unra1}, changing variables 
in $x$ and using the invariant properties of the functions involved, we obtain
\begin{equation}\label{unra2}
\gamma(a)\,|ab^2|^{-1/2}\int\limits_{V_1}
\int\limits_{Y_0}f_W^{(n)}(y_0w_0^*\gamma_0vw_0g)\,dy_0\,dv.\notag
\end{equation}
Here $\gamma(a)$ is the Weil factor and the factor $\gamma(a)\,|ab^2|^{-1/2}$ is obtained from the change
of variables in $x$ together with the properties of the Weil representation. We can ignore the
integration over $x$ since $\phi$ is the unramified Schwartz function. 

The matrices $\gamma_0$ and $\gamma$ were defined before \eqref{wh111} above. We have
the factorization
$\gamma=\begin{pmatrix} 1&1\\ &1\end{pmatrix}\begin{pmatrix} &-1\\ 1&\end{pmatrix}
\begin{pmatrix} 1&\tfrac12\\ &1\end{pmatrix},$ and this induces a factorization of $\gamma_0$:
$\gamma_0=\gamma_0'w_0'\gamma_0''$. Changing variables in $V_1$ we may ignore the matrix $\gamma_0''$.
Moving $\gamma_0'$ to the left and using the left transformation of $f_W^{(n)}$ by matrices
in $U_{2n,n}$, we obtain the integral
\begin{equation}\label{unra3}
\gamma(a)\,|ab^2|^{-1/2}\int\limits_{V_1}
\int\limits_{Y_0}f_W^{(n)}(y_0w_0^*w_0'vw_0g)\,\psi_{Y_0}(y_0)\,dy_0\,dv.\notag
\end{equation}
Here we recall that $n$ is odd and that $Y_0$ consists of all matrices of
the form
$$y_0=\begin{pmatrix} 1&&&\\ &I_{n-1}&&\\ &y&I_{n-1}&\\ &&&1\end{pmatrix}\quad
\text{with~} y\in \text{Mat}_{(n-1)\times (n-1)}, y_{i,j}=0~\text{for all~} i\ge j;$$
we have $\psi_{Y_0}(y_0)=\psi(y_{\frac{n-1}{2},\frac{n+1}{2}})$. 

Next we move the matrix $g$ to the left and the Weyl element $w_0^*w_0'w_0$ to the right.  Doing so, we
obtain the integral 
\begin{equation}\label{unra4}
\gamma(a)\,|ab^2|^{-\frac{n-2}{2}}\int\limits_{V_2}
\int\limits_{Y_0}f_W^{(n)}(ty_0v_2)\,\psi_{Y_0}(y_0)\,dy_0\,dv_2
\end{equation}
with $t=\text{diag}(ab,b,I_{2n-4},b^{-1},a^{-1}b^{-1})$.   The group
$V_2$ consists of all matrices of the form $\left(\begin{smallmatrix} I_n&\\ v&I_n\end{smallmatrix}\right)$ where
$v=\sum_{i,j} r_{i,j}(e_{i,j}+e_{n-j+1,n-i+1})$ and the sum is over all pairs $(i,j)$ such that $1\le i\le 
\frac{n-3}{2}$ and $1\le j\le i+1$. The extra factor of 
$|ab^2|^{-\frac{n-3}{2}}$ is obtained from the change of variables in $V_1$.

Integral \eqref{unra4} is the analogous to the integral that appears on the right hand side of
Ginzburg-Rallis-Soudry \cite{G-R-S2}, p.\ 255, Eqn.\ (3.7).  Indeed, the matrices $y_0$ and $v_2$ can be viewed
as matrices of $Sp_{2n-2}$ embedded in $Sp_{2n}$ in the obvious way. This comparison enables 
us to proceed as in \cite{G-R-S2}, pp.\ 256-260, and conclude that integral \eqref{unra4} is a sum
of two terms. First we have 
\begin{equation*}\label{unra5}
\gamma(a)\,|ab^2|^{-\frac{n-2}{2}}f_W^{(n)}(t)=\gamma(a)|ab^2|^{\frac{2n-1}{2n}}W_{GL_n}^{(n)}(t_0).
\end{equation*}
This corresponds to the term which is denoted as Case 1 in \cite{G-R-S2}, p.\ 260. The second term
we obtain is the one corresponding to Case 2 in that reference. It is equal to
\begin{equation}\label{unra6}
\gamma(a)|ab^2|^{-\frac{n-2}{2}}\int\limits_{|h|>1}f_W^{(n)}(ty(h))\,\psi(h)\,|h|^{\frac{(n-1)(n-3)}
{2}}\,dh,
\end{equation}
with $y(h)=\text{diag}(I_2,h^{-1}I_{n-2},hI_{n-2},I_2)$. Write $h=p^{-m}\epsilon$ where $\epsilon$
is a unit. Then \eqref{unra6} is equal to
\begin{equation}\label{unra70}
\gamma(a)\,|ab^2|^{-\frac{n-2}{2}}\sum_{m=1}^\infty q^{m\left(1+\frac{(n-1)(n-3)}{2}\right)}\int\limits_{
|\epsilon|=1}f_W^{(n)}(ty(p^{-m}\epsilon))\,\psi(p^{-m}\epsilon)\,d\epsilon.\notag
\end{equation}
The factorization $y(p^{-m}\epsilon)=y(p^{-m})y(\epsilon)$ contributes, via the cocycle defining the metaplectic group, the factor
$(\epsilon,p)_n^{m(n-2)}$. Hence the above integral is equal to
\begin{equation}\label{unra7}
\gamma(a)\,|ab^2|^{-\frac{n-2}{2}}\sum_{m=1}^\infty q^{m\left(1+\frac{(n-1)(n-3)}{2}\right)}f_W^{(n)}(ty(p^{-m}))
\int\limits_{|\epsilon|=1}(\epsilon,p)_n^{m(n-2)}\,\psi(p^{-m}\epsilon)\,d\epsilon.
\end{equation}
The inner integral is zero unless $m=1$, and when $m=1$ it is equal to $q^{-1/2}G_{n-2}^{(n)}(p)$
where the last term is a normalized $n$-th order Gauss sum, as in \cite{F-G}, Section 6. Substituting this into \eqref{unra7} we
obtain 
\begin{equation}\label{unra8}
\gamma(a)\,|ab^2|^{-\frac{n-2}{2}}q^{-\frac{(n-2)(2n-1)}{2n})}W_{GL_n}^{(n)}(t_0y(p)_0).\notag
\end{equation}

Thus we have established the following result, relating the Whittaker coefficients of the descent on a suitable
cover of the symplectic group to the Whittaker coefficients of the theta function on a cover of the general linear group.
\begin{theorem}\label{th2}
Let $n\ge 3$ be an odd integer.  Let  $a=p^{n_1}$, $b=p^{n_2}$, with $n_1,n_2\geq0$.  Let
$$t=\text{\rm diag}(ab,b,I_{n-3},b^{-1},a^{-1}b^{-1})\in Sp_{n+1}(F),$$
$$t_0=\text{\rm diag}(ab,b,I_{n-2})\in GL_n(F),\quad t_1=\text{\rm diag}(ab,b,pI_{n-2})\in GL_n(F).$$
 Suppose that Conjecture~\ref{conj1}$_{n,i}$ holds for $(n-1)/2\leq i\leq n$.  Then we have the identity
\begin{equation}\label{unra9}
W_{Sp_{n+1}}^{(2n)}(t)=\gamma(a)\,|ab^2|^{\frac{2n-1}{2n}}\left(W_{GL_n}^{(n)}(t_0)+
q^{-\frac{(n-2)(2n-1)}{2n}}W_{GL_n}^{(n)}(t_1)\right).
\end{equation}
\end{theorem}

\section{A Case of Conjecture~\ref{conj1}}

Suppose $n=3$.  Then Conjecture~\ref{conj1}$_{3,i}$ is known for $i=1$ (\cite{K-P}, Corollary II.2.6) and $i=2$ (\cite{F-G}, 
Proposition 3).  
In this section we sketch the proof of Conjecture~\ref{conj1}$_{3,3}$.
\begin{proposition}\label{conjprop1}
Conjecture \ref{conj1}$_{3,3}$ holds. That is,  ${\mathcal O}(\Theta_{6}^{(3)})=(3^2)$.
\end{proposition}

As a consequence, the hypotheses of Theorem~\ref{th2} are satisfied for $n=3$, and we conclude unconditionally the
Whittaker function for the descent on the $6$-fold cover of $Sp_4$ is Eulerian, and 
given at unramified places by formula \eqref{unra9}.

\begin{proof}
To prove the Proposition we need to prove that the representation $\Theta_{6}^{(3)}$ has no nonzero Fourier coefficients 
corresponding to the unipotent orbits $(6)$ and $(42)$. Once we do so,  an argument similar to the one given in 
Ginzburg \cite{G1}, Thm.\ 3.1, will imply that this representation also has no nonzero Fourier coefficients corresponding to the orbit $(41^2)$. We also need to show that $\Theta_{6}^{(3)}$  has a nonzero Fourier coefficient
corresponding to the unipotent orbit $(3^2)$. This last follows from Proposition~\ref{propconstant}.  So the critical step is the
vanishing of Fourier coefficients attached to the two unipotent orbits $(6)$ and $(42)$.

For convenience, we drop the prior notation for unipotent subgroups and characters, and will introduce the ones needed anew.

To prove this claim for $\Theta_{6}^{(3)}$ we make use of the fact that this representation can be obtained as a residue of an Eisenstein series induced from a  suitable theta function of a smaller rank group. This is explained in Section~\ref{def} above. 
First we consider the unipotent orbit $(6)$.
Since $\Theta_{6}^{(3)}$ is a residue of $E_{\Theta_{4}}^{(3)}(g,s)$, in order to prove that
$\Theta_{6}^{(3)}$ is not generic it is enough to prove that $E_{\Theta_{4}}^{(3)}(g,s)$ is not generic. Moreover, let $U$ denote the standard maximal unipotent  of $Sp_6$, and let $\psi_U$ denote a Whittaker  character of $U$;
without loss we take $\psi_U(u)=\psi(u_{1,2}+u_{2,3}+au_{3,4})$ with
$a\in F^*$. Then a standard unfolding argument shows that the integral
$$\int\limits_{U(F)\backslash U({
\A})}E_{\Theta_{4}}^{(3)}(ug,s)\,\psi_U(u)\,du$$ is not zero for some choice of data if and only if the representation
$\Theta_{4}^{(3)}$ is generic. However, it follows from Friedberg-Ginzburg
\cite{F-G}, Lemma 2, that this representation is not generic. Hence, for $\text{Re}(s)$ large,  the above integral is zero for all choices of data, and by continuation it follows that $\Theta_{6}^{(3)}$ is not generic.

Next we consider the Fourier coefficients corresponding to the unipotent orbit $(42)$. Let $V\subseteq U$ denote the subgroup 
$V=\{ u=(u_{i,j})\in U\ :\ u_{2,3}=u_{4,5}=0\}$. This is the unipotent radical of the parabolic subgroup of $Sp_6$ whose Levi part is $L=GL_1\times GL_2$. The embedding of $L$ into $Sp_6$ is given by $L=\{
\text{diag}(g_1,g_2,g_2^*,g_1^{-1})\ :\ g_1\in GL_1;\ g_2\in
GL_2\}$. Let $\alpha_1,\alpha_2\in F$ and
$\alpha_3,\alpha_4\in F^*$. Consider the Fourier coefficient defined by
\begin{equation}\label{421}
\int\limits_{V(F)\backslash V({
\A})}\varphi(vg)\,\psi_{V,\alpha_i}(v)\,dv,
\end{equation}
where $\psi_{V,\alpha_i}(v)=\psi(\alpha_1
v_{1,2}+\alpha_2 v_{1,3}+ \alpha_3 v_{3,4}+ \alpha_4 v_{2,5})$.
This Fourier coefficient is
associated with the unipotent orbit $(42)$ provided the connected component of the stabilizer of $\psi_{V,\alpha_i}$ inside
$L(F)$ is trivial (see for example Ginzburg \cite{G1}).
If $\alpha_3\alpha_4=-\epsilon^2$ for some $\epsilon\in F^*$, then one can find a suitable discrete element $\gamma_L\in L$ 
and $\alpha \in F^*$ such that the integral \eqref{421} is equal to
\begin{equation}\label{4211}
\int\limits_{V(F)\backslash V({\A})}\varphi(v\gamma_L
g)\,\widetilde{\psi}_{V,\alpha}(v)\,dv,
\end{equation}
where $\widetilde{\psi}_{V,\alpha}(v)=\psi(\alpha v_{1,2}+v_{1,3}+v_{2,4})$.
We will say that the Fourier coefficient \eqref{421} is of split type if $\alpha_3\alpha_4=-\epsilon^2$  for some $\epsilon\in F^*$, 
and non-split otherwise.

Since at half of the places $\nu$ we have the condition $\alpha_3\alpha_4=-\epsilon^2$ for some $\epsilon\in F_\nu^*$, 
 to complete the proof it is enough to
prove that for every finite local unramified place, the Jacquet 
module $J_{V,\widetilde{\psi}_{V,\alpha}}(\Theta_6^{(3)})$ is zero for all characters $\widetilde{\psi}_{V,\alpha}$
with $\alpha\in F^*$, where $\Theta_6^{(3)}$ now denotes the local theta representation at $\nu$.
To do so we consider the Jacquet module associated with the unipotent orbit $(3^2)$. 
Let $R$ denote the subgroup of $U$ defined by $R=\{ u\in U\ :\
u_{1,2}=u_{3,4}=u_{5,6}=0\}$. 
Then $R$ is the unipotent radical of the maximal parabolic subgroup  
of $Sp_6$ whose Levi part is $GL_2\times SL_2$.
Let $\psi_R$ denote the character $\psi_R(r)=\psi(r_{1,3}+r_{2,4})$ of $R$ (that is, of $R(F_\nu)$; from now on we take the $F_\nu$
points of the algebraic groups in question without changing the notation).

Next, we compute 
$J_{RN,\psi_R\psi_{\alpha}}(\Theta_6^{(3)})$, where $N$ is the unipotent radical of $SL_2$ embedded in $Sp_6$ as the group of all matrices
$$n(x)=\begin{pmatrix} 1&x&&&&\\ &1&&&&\\ &&1&x&&\\ &&&1&&\\
&&&&1&-x\\ &&&&&1\end{pmatrix},$$ 
and $\psi_{\alpha}$ is the
character of $N$ defined by $\psi_\alpha(n(x))=\psi(\alpha x)$, with
$\alpha\in F_\nu$. 
Using the one-parameter subgroups $\{z(k)=I_6+ke_{3,4}\}$ and $\{y(m)=I_6+me_{2,3}-me_{4,5}\}$,
we perform root exchanges using the Lemma in Section 2 of \cite{G-R-S1}.
This gives the isomorphism of Jacquet modules
\begin{equation}\label{jacquet-isom}
J_{RN,\psi_R\psi_{\alpha}}(\Theta_6^{(3)})\cong
J_{V,\widetilde{\psi}_{V,\alpha}}(\Theta_6^{(3)}).
\end{equation}
Thus, to prove Prop.~\ref{conjprop1}, it is enough to prove that for all
$\alpha\neq0$, the Jacquet module $J_{RN,\psi_R\psi_{\alpha}}(\Theta_6^{(3)})$ is zero.

 It is not hard to check that the stabilizer of $\psi_R$ inside $GL_2\times SL_2$ is $SL_2$ embedded diagonally. Thus, 
$J_{R,\psi_R}(\Theta_6^{(3)})$, if not zero, defines an unramified
representation of $SL_2$. 
Let $\sigma$ denote an irreducible unramified representation of 
$SL_2$ such that the space
$$\text{Hom}_{SL_2}(\sigma,J_{R,\psi_R}(\Theta_6^{(3)}))$$ 
is not zero. Since $\sigma$ is unramified it is a sub-representation
of $\Ind_{B}^{SL_2}\chi\delta_B^{1/2}$. Applying Frobenius reciprocity
we deduce that the space 
$$\text{Hom}_{GL_1}(\chi,J_{RN,\psi_R}(\Theta_6^{(3)}))$$
is not zero. We shall show below that we must have $\chi=1$, which implies that $\sigma=1$.
Since the Jacquet module is always of finite length, we then conclude that all Jacquet modules 
$J_{RN,\psi_R\psi_{\alpha}}(\Theta_6^{(3)})$ with $\alpha\ne 0$ are
zero. Thus once we have established that $\chi$ must be $1$ (and in particular,
that $J_{R,\psi_R}(\Theta_6^{(3)})$ is nonzero), the Proposition will be proved.

To do so, we use the isomorphism \eqref{jacquet-isom} above with $\alpha=0$.
If $\alpha=0$, then the Jacquet module $J_{RN,\psi_R\psi_{0}}(\Theta_6^{(3)})=J_{RN,\psi_R}(\Theta_6^{(3)})\cong J_{V,\psi_V}(\Theta_6^{(3)})$, if not zero, defines an unramified character of $GL_1$, which is embedded in $Sp_6$ as the
group of all matrices $\{t(a)\}$, with $t(a)={\text{diag}}(a,a^{-1},a,a^{-1},a,a^{-1})$. 
Here $\psi_V=\widetilde{\psi}_{V,\alpha}$ with $\alpha=0$.
For $1\le i\le 3$, let $w_i$ denote the three simple reflections of the  Weyl group of $Sp_6$. In matrices we have
$$w_1=\begin{pmatrix} &1&&&&\\ 1&&&&&\\ &&1&&&\\ &&&1&&\\
&&&&&1\\ &&&&1&\end{pmatrix};\ \ w_2=\begin{pmatrix} 1&&&&&\\ &&1&&&\\ &1&&&&\\ &&&&1&\\
&&&1&&\\ &&&&&1\end{pmatrix};\ \ w_3=\begin{pmatrix} 1&&&&&\\ &1&&&&\\ &&&1&&\\ &&-1&&&\\
&&&&1&\\ &&&&&1\end{pmatrix}.$$ 
Let $Y=\{y_1(m_1)=I_6+m_1e_{4,3}\}$, and let $U_1$ be the subgroup of the maximal unipotent subgroup $U$ of $Sp_6$ consisting 
of all matrices $u=(u_{i,j})\in U$ such that $u_{2,4}=u_{3,5}=u_{3,4}=0$. Let $\psi_{U_1}(u_1)=\psi(u_1(1,2)+u_1(2,3))$. Then 
$U_1Y=(w_3w_2)V(w_3w_2)^{-1}$. 
Hence, $J_{V,\psi_V}(\Theta_6^{(3)})$ is isomorphic to 
$J_{U_1,\psi_{U_1}}(\Theta_6^{(3)})$. Let $X=\{I_6+ke_{2,4}+ke_{3,5}\}$, and perform a root exchange of this unipotent subgroup with the 
group $Y$. Then, since $\Theta_6^{(3)}$ is not generic,
we obtain that $J_{U_1,\psi_{U_1}}(\Theta_6^{(3)})$ is isomorphic 
to $J_{U,\psi_{U}}(\Theta_6^{(3)})$. 
This Jacquet module is not zero. Indeed, using an argument similar to
\cite{B-F-G} Theorem 2.3, the nonvanishing of this Jacquet module
follows from the genericity of the theta representation of $GL_3^{(3)}$. 
From this we deduce that the Jacquet module 
$J_{R,\psi_R}(\Theta_6^{(3)})$ is not zero.
Moreover, the Jacquet module $J_{U,\psi_{U}}(\Theta_6^{(3)})$ acts by the character $|a|^4$ under the torus
$\text{diag}(aI_3,a^{-1}I_3)\in Sp_6$. Taking into an account the 
various root exchanges, we deduce that the group 
$GL_1=\{t(a)\}$ acts trivially on the Jacquet module $J_{RN,\psi_R}(\Theta_6^{(3)})$. 
Thus $\chi$ above must be $1$.
This completes the proof of the Proposition.
\end{proof}

\end{document}